\documentclass[12pt, reqno]{amsart}
\usepackage{amsmath, amsthm, amscd, amsfonts, amssymb, graphicx, color}
\usepackage{mathrsfs}
\usepackage{graphicx}
\usepackage{caption}
\usepackage{subcaption}
\usepackage[bookmarksnumbered, colorlinks, plainpages]{hyperref}
\hypersetup{colorlinks=true,linkcolor=red, anchorcolor=green, citecolor=cyan, urlcolor=red, filecolor=magenta, pdftoolbar=true}

\newtheorem{theorem}{Theorem}[section]
\newtheorem{lemma}[theorem]{lemma}

\theoremstyle{remark}
\newtheorem{remark}[theorem]{\bf{Remark}}
\numberwithin{equation}{section}
\allowdisplaybreaks
\begin{document}

\title[]{Logarithmic Inverse Coefficients and moduli differences of Janowski convex class}

\author[C. Dhara and  N. Ghosh]{Chayani Dhara and Nirupam Ghosh}
	
\address[Dhara]{Department of Mathematics, Indian Institute of Engineering Science and Technology, Shibpur, Howrah 711103, West Bengal, India}
\email{chayanidhara1999@gmail.com}

\address[Ghosh] {Department of Mathematics, Indian Institute of Engineering Science and Technology, Shibpur, Howrah 711103, West Bengal, India}
\email{nirupamghoshmath@gmail.com}

\subjclass[2020]{30C45, 30C50, 30C80}

\keywords{Univalent, Convex functions, Differential subordination, Schwarz function, Coefficient estimates, Logarithmic inverse coefficient, Hankel determinant.}

\begin{abstract} 
	 In this paper, we study the sharp bounds of the first three logarithmic inverse coefficients for Janowski convex class $\mathcal{C}(A, B)$. We also derive sharp upper and lower bounds of $\bigl|\,\gamma_2 \,\bigr|-\bigl|\,\gamma_1\,\bigr|$ and $\bigl|\,\Gamma_2 \,\bigr|-\bigl|\,\Gamma_1\,\bigr|$ for functions in the class $\mathcal{C}(A, B)$. Furthermore, a sharp estimate for the second Hankel determinant associated with the logarithmic inverse coefficients for functions in $\mathcal{C}(A, B)$  is obtained.
\end{abstract}

\maketitle	

\section{Introduction}
Let  $\mathcal{H}$ be the class of holomorphic functions defined on the unit disk $\mathbb{D}:=\{z\in\mathbb{C}:|z|<1\}$. Let $\mathcal{A}$ be the class of functions $f\in\mathcal{H}$ such that $f(0) = 0 $ and $ f'(0) = 1 $. Any functions $f \in \mathcal{A}$ have the following power series expansion
\begin{equation}\label{eq:1.1}
f(z)=z+\sum_{n=2}^{\infty} a_n z^n, \quad z \in \mathbb{D}.
\end{equation}
 The most familiar subclass of $\mathcal{A}$ consists of univalent (i.e, one-to-one) functions $f$ in $\mathbb{D}$, and it is denoted by $\mathcal{S}$. 
Let $\mathcal{S}^*$ and $\mathcal{C}$ be the classes of starlike and convex functions in $\mathcal{S}$, respectively. A function $f\in\mathcal{A}$ is said to be a starlike function in $\mathbb{D}$ if $f(\mathbb{D})$ is a starlike domain with respect to the origin
and $f\in\mathcal{A}$ is said to be  convex if $ f(\mathbb{D})$ is a convex domain.

For $f, g \in \mathcal{H}$, we say that  $f$ is subordinate to $g$, written as
$f \prec g$, if there exists an analytic function $\omega : \mathbb{D} \to \mathbb{D}$
satisfying $\omega(0) = 0$ such that
$$f(z) = g(\omega(z)), \quad z \in \mathbb{D}.$$
In particular, if $g$ is univalent in $\mathbb{D}$, then the subordination $f \prec g$ holds if and only if $f(0) = g(0)$ and $f(\mathbb{D}) \subset g(\mathbb{D})$.
For $\varphi \in \mathcal{H}$ with $\varphi(0)=1$, let the classes $\mathcal{S}^*(\phi)$ and  $\mathcal{C}(\varphi)$ be defined by $$\mathcal{S}^*(\phi) = \left\{ f \in \mathcal{A} :
\frac{z f'(z)}{f(z)} \prec \phi(z) \right\},~~~~~~~~~~~~\mathcal{C}(\phi) = \left\{ f \in \mathcal{A} :
1 + \frac{z f''(z)}{f'(z)} \prec \phi(z) \right\}.$$
This class was introduced by Ma and Minda~\cite{MaMinda1992} under the assumptions that $\varphi$ is univalent in $\mathbb{D}$, has positive real part, and maps $\mathbb{D}$ onto a domain symmetric with respect to the real axis and starlike with respect to the point $\varphi(0) = 1$, with $\varphi'(0) > 0$.
These two classes are referred to as the Ma--Minda classes of starlike and convex functions, respectively.
It is worth noting that 
$f \in \mathcal{C}(\phi)$ if and only if $z f'(z) \in \mathcal{S}^*(\varphi).$
By choosing appropriate functions $\varphi$, these general classes reduce to many well-known subclasses studied in geometric function theory.
For example, if $\varphi(z) = {(1+z)}/{(1-z)},$
then $\mathcal{S}^*(\varphi)$ and $\mathcal{C}(\varphi)$ reduce to the classical classes $\mathcal{S}^*$ and $\mathcal{C}$ of starlike and convex functions respectively.
If $\varphi(z) = \bigl(1+(1-2\alpha)z\bigr)/(1-z), 0 \le \alpha < 1$ then the corresponding classes become $\mathcal{S}^*(\alpha)$ and $\mathcal{C}(\alpha)$ representing Starlike class of order $\alpha$ and convex class of order $\alpha$ respectively. 
Further, if $\phi(z) = {(1+Az)}/{(1+Bz)}, -1 \le B < A \le 1, $ then we obtain the Janowski class of starlike functions $\mathcal{S}^*(A,B)$ and convex functions $\mathcal{C}(A,B)$ respectively. Significant results about these classes can be found in \cite{Janowski1973a, Janowski1973b}. 

For each $f \in \mathcal{S}$ there exists a unique function $F_{f}\in\mathcal{A} $ such that $ F_{f}(0)=0$ and $$f(z)=z\exp(F_{f}(z)), ~~~z \in \mathbb{D}.$$ 
 Moreover, for some neighborhood of the origin, $F_{f}$ can be expressed as 
 \begin{equation}\label{P2-eq-01A}
   F_{f}(z)=\log \frac{f(z)}{z} = 2\sum_{n=1}^{\infty} \gamma_n z^n,   
 \end{equation}
where $\gamma_n,~n \in \mathbb{N}$ are known as logarithmic coefficients of $f \in \mathcal{S}$.
Note that differentiating (\ref{P2-eq-01A}) and using (\ref{eq:1.1}), we get
\begin{align}\label{eq:gamma-coefficients}
    \gamma_1 &= \frac{1}{2}a_2, \nonumber \\
    \gamma_2 &= \frac{1}{2}\left(a_3 - \frac{1}{2}a_2^2\right), \nonumber \\
    \gamma_3 &= \frac{1}{2}\left(a_4 - a_2 a_3 + \frac{1}{3}a_2^3\right). 
\end{align}

As $|a_2| \leq 2$, for $f \in \mathcal{S}$, $|\gamma_1| \leq 1$. Using the Fekete-Szeg\"o inequality \cite [Theorem 3.8]{duren1983univalent} for $ f \in \mathcal{S}$, we can get the sharp estimate 
$|\gamma_{2}| \leq \frac{1}{2} (1 + 2 e^{-1})| = 0.635\ldots.$
For $n \geq 3$, no significant upper bound for $|\gamma_n|$ where $f \in \mathcal{S}$ is yet to be known. Logarithmic coefficients for the class $\mathcal{S}$ and its subclasses are a recent topic of interest for various authors. Several important contributions to this topic have appeared in 
\cite{Firoz2018, Duren1979, Elhosh1996, Girela2000, LeckoSim2025}.
Recently, Ali and Thakur \cite{Ali2026}  obtained  the sharp estimates of the first three logarithmic coefficients for functions in the class $\mathcal{C}( A, B )$.\\

The notion of the logarithmic coefficients of the inverse of $f$ was introduced by Ponnusamy et al. \cite{Ponnusamy2018}. The logarithmic inverse coefficients $\Gamma_n, n \in \mathbb{N}$, of $f$ are defined by the equation 
\begin{equation}\label{eq:log-inverse-def}
    \log \frac{f^{-1}(w)}{w}
    = 2\sum_{n=1}^{\infty} \Gamma_n\, w^n,
    \qquad |w|<\frac{1}{4}.
\end{equation}
In \cite{Ponnusamy2018}, Ponnusamy et al. obtained a sharp bound for the logarithmic inverse coefficients for the class $\mathcal{S}$. In fact, Ponnusamy et al. \cite{Ponnusamy2018} established for $f \in \mathcal{S}$ that $|\Gamma_n| \le \frac{1}{2n}\binom{2n}{n},$
and showed that the equality holds only for the Köebe function or its rotations. Moreover, Ponnusamy et al. \cite{Ponnusamy2018} obtained a sharp bound for logarithmic inverse coefficients for some of the important geometric subclasses of $\mathcal{S}$. It is also known that  for $f \in \mathcal{A}$  defined by \eqref{eq:1.1}, logarithmic inverse coefficients and Taylor's coefficients are related as 
\begin{align}\label{eq:Gamma-coefficients}
    \Gamma_1  = -\frac{1}{2}a_2,~~
    \Gamma_2  = -\frac{1}{2}\left(a_3 - \frac{3}{2}a_2^2\right)~~ \mbox{and} ~~~
    \Gamma_3  = -\frac{1}{2}\left(a_4 - 4a_2 a_3 + \frac{10}{3}a_2^3\right). 
\end{align}
In 2023, as a consequence of  Loewner theory, Lecko and Partyka~\cite{LeckoPartyka2023} investigated sharp upper and lower bounds for $|\gamma_2| - |\gamma_1|$ for functions in the class $\mathcal{S}$. Subsequently, Obradović and Tuneski~\cite{ObradovicTuneski2023} presented a simpler proof of this result. Moreover Kumar and Cho~\cite{KumarCho2023} established sharp bounds for $\bigl|\,\gamma_2 \,\bigr|-\bigl|\,\gamma_1\,\bigr|$ within several subclasses of $\mathcal{S}$. More recently, Allu and Shaji \cite{Allu2025} estimated the sharp upper and lower bounds for moduli difference $\bigl|\,\Gamma_2 \,\bigr|-\bigl|\,\Gamma_1\,\bigr|$ for functions in class $\mathcal{S}$ and for functions in some important subclasses of univalent functions.\\[4mm]
For $q, n \in \mathbb{N}$, the Hankel determinant $H_{q,n}(f)$ of the Taylor's coefficients of a
function $f \in \mathcal{A}$ of the form \eqref{eq:1.1} is defined by
$$
H_{q,n}(f) =
\begin{vmatrix}
a_n & a_{n+1} & \cdots & a_{n+q-1} \\
a_{n+1} & a_{n+2} & \cdots & a_{n+q} \\
\vdots & \vdots & \ddots & \vdots \\
a_{n+q-1} & a_{n+q} & \cdots & a_{n+2(q-1)}
\end{vmatrix}.
$$
Hankel determinants of various orders have recently attracted considerable attention and have been investigated by several authors (see \cite{ref12, Krishna2015, ref5, ref18, Sim2022}). It is well known that the Fekete-Szegö functional corresponds to the second Hankel determinant $H_{2,1}(f)$. Fekete-Szegö then further generalized  by establishing bounds for  $\lvert a_3 - \mu a_2^2 \rvert$
with $\mu \in \mathbb{R}$  for $f \in \mathcal{S}$ (see \cite[Theorem 3.8]{duren1983univalent}). Recently increasing attention has been devoted to the investigation of Hankel determinants involving logarithmic coefficients (denoted by $H_{q,n}(F_f/2) $)  and logarithmic inverse coefficients (denoted by $H_{q,n}(F_{f^{-1}}/2)$) for various geometric subclasses of $\mathcal{S}$ (see \cite{ref3, ref6, ref13,ref14, SunWang2025}).\\

In \ref{section-2}, we collect several auxiliary results and lemmas that are required in the proofs of the main theorems. \ref{section-3} focuses on obtaining sharp bounds for the logarithmic inverse coefficients of functions in the Janowski convex class $\mathcal{C}(A, B)$. As consequences of these results, we recover several known estimates for the class $\mathcal{C}(\alpha)$ and obtain new sharp bounds that complete the determination of $|\Gamma_3(F)|$ for all $\alpha\in[0,1)$.  In \ref{section-4}, we establish sharp upper and lower bounds for the differences of the moduli of successive logarithmic coefficients and logarithmic inverse coefficients associated with functions in $\mathcal{C}(A, B)$. Finally, in \ref{section-5}, we derive sharp estimates for the second Hankel determinant formed by the logarithmic inverse coefficients of functions belonging to $\mathcal{C}(A, B)$.

\section{Preliminary Results}\label{section-2}

In this section, we present several known results from the literature that
will be instrumental in establishing our main theorems.\\
Let $\mathcal{B}$ be the class of all analytic functions on $\mathbb{D}$ satisfying $|\omega(z)| < 1$.   Also  $\mathcal{B}_0$ is a geometric subclass of $\mathcal{B}$ defined by the normalization $\omega(0)=0$. Any function  $\omega \in \mathcal{B}_0$ have the following power series expansion 
\begin{equation}\label{P2-eq-001}
\omega(z) = \sum_{n=1}^{\infty} c_n z^n. 
\end{equation}

The following lemmas for functions in the class $\mathcal{B}_0$ are essential to prove our results. 

\begin{lemma}\cite{Carlson1940} \label{p2-lemma-001}
Let $\omega \in \mathcal{B}_0$ be of the form (\ref{P2-eq-001}). Then for any $\lambda \in \mathbb{C}$,
the following sharp estimates hold:
$$|c_2 + \lambda c_1^2| \le \max\{1, |\lambda|\}.$$
\end{lemma}

\begin{lemma}\cite{Libera1982}\label{lemma-4}
Let $\omega \in \mathcal{B}_0$ be of the form (\ref{P2-eq-001}). Then for some $x,s \in \mathbb{C}$ with $|x|\le1$ and $|s|\le1$
\begin{align*}
& c_2=(1-c_1^2)x,\\
& c_3=(1-c_1^2)(1-|x|^2)s-(1-c_1^2)c_1x^2.
\end{align*}
\end{lemma}


\begin{lemma}\cite{ProkhorovSzynal1981} \label{p2-lemma-002}
Let $\omega \in \mathcal{B}_0$ be of the form (\ref{P2-eq-001}). Then, for any real parameters
$\mu$ and $\nu$, the following inequality holds:
$$
|c_3 + \mu c_1 c_2 + \nu c_1^3| \le
\begin{cases}
1, & (\mu,\nu) \in D_1 \cup D_2 \cup \{(2,1)\}, \\[6pt]
|\nu|, & (\mu,\nu) \in \displaystyle\bigcup_{k=3}^{7} D_k, \\[8pt]
\dfrac{2}{3}(|\mu|+1)\left(\dfrac{|\mu|+1}{3(|\mu|+1+\nu)}\right)^{1/2},
& (\mu,\nu) \in D_8 \cup D_9, \\[10pt]
\dfrac{1}{3}\nu \left(\dfrac{\mu}{\mu^2 - 4\nu}\right)
\left(\dfrac{\mu^2 - 4}{3(\nu-1)}\right)^{1/2},
& (\mu,\nu) \in D_{10} \cup D_{11} \setminus \{(2,1)\}, \\[10pt]
\dfrac{2}{3}(|\mu|-1)\left(\dfrac{|\mu|-1}{3(|\mu|-1-\nu)}\right)^{1/2},
& (\mu,\nu) \in D_{12}.
\end{cases}
$$
The regions $D_k  ~~ \mbox{for}~~~k=1,\dots,12$ are defined as follows:
\begin{align*}
& D_1=\{(\mu,\nu):|\mu|\le \tfrac12,\; \nu \le 1\},\\
& D_2=\{(\mu,\nu):\tfrac12\le|\mu|\le 2,\; \tfrac{4}{27}(|\mu|+1)^3-(|\mu|+1)\le \nu\le 1\},\\
& D_3=\{(\mu,\nu):|\mu|\le\tfrac12,\;\nu\le -1\},\\
& D_4=\{(\mu,\nu):|\mu|\ge\tfrac12,\;\nu\le -\tfrac23(|\mu|+1)\},\\
& D_5=\{(\mu,\nu):|\mu|\le 2,\; \nu\ge 1\},\\
& D_6=\{(\mu,\nu):2\le|\mu|\le 4,\; \nu \ge \tfrac1{12}\sqrt{\mu^2+8}\},\\
& D_7=\{(\mu,\nu):|\mu|\ge 4,\;|\nu|\ge \tfrac23(|\mu|-1)\},\\
& D_8 = \left\{ (\mu,\nu) : \tfrac{1}{2} \le |\mu| \le 2,\ 
-\tfrac{2}{3}(|\mu|+1) \le \nu \le
\frac{4}{27}(|\mu|+1)^3 - (|\mu|+1) \right\},\\
& D_9 = \left\{ (\mu,\nu) : |\mu| \ge 2,\ 
-\tfrac{2}{3}(|\mu|+1) \le \nu \le
\frac{2|\mu|(|\mu|+1)}{\mu^2 + 2|\mu| + 4} \right\},\\
& D_{10} = \left\{ (\mu,\nu) : 2 \le |\mu| \le 4,\ 
\frac{2|\mu|(|\mu|+1)}{\mu^2 + 2|\mu| + 4}
\le \nu \le \tfrac{1}{12}(\mu^2 + 2) \right\},\\
& D_{11} = \left\{ (\mu,\nu) : |\mu| \ge 4,\ 
\frac{2|\mu|(|\mu|+1)}{\mu^2 + 2|\mu| + 4}
\le \nu \le\frac{2|\mu|(|\mu|-1)}{\mu^2 - 2|\mu| + 4} \right\},\\
& D_{12} = \left\{ (\mu,\nu) : |\mu| \ge 4,\ 
\frac{2|\mu|(|\mu|-1)}{\mu^2 - 2|\mu| + 4}
\le \nu \le \tfrac{2}{3}(|\mu|-1) \right\}.
\end{align*}

Furthermore, all the above inequalities are sharp.
\end{lemma}
Let $\mathcal{P}$ be the class of all analytic functions $p$ in the unit disk 
$\mathbb{D} = \{ z \in \mathbb{C} : |z| < 1 \}$ satisfying $p(0) = 1$ and $Re\, p(z) > 0$ for $z \in \mathbb{D}$.
Therefore, every $p \in \mathcal{P}$ can be represented as 
\begin{equation}\label{P Class-1}
 p(z) = 1 +\sum_{n=1}^{\infty} p_n z^n \quad z \in \mathbb{D}.   
\end{equation}
Elements of the class $\mathcal{P}$ are called \textit{Carath\'eodory functions}. 
\begin{lemma}\cite{SimThomas2020}\label{p1-lemma-003}
Let $p \in \mathcal{P}$ be of the form (\ref{P Class-1}) and  $B_1, B_2,$ and $B_3$ be numbers such that $B_1 \ge 0$,  $B_2 \in \mathbb{C}$, and $B_3 \in \mathbb{R}$. 
Define $\Psi_{+}(p_1,p_2)$ and $\Psi_{-}(p_1,p_2)$ by
$$
\Psi_{+}(p_1,p_2) = \left| B_2 p_1^2 + B_3 p_2 \right| - \left| B_1 p_1 \right|,
$$
and
$$
\Psi_{-}(p_1,p_2) = -\Psi_{+}(p_1,p_2).
$$
Then
$$
\Psi_{+}(p_1,p_2) \le
\begin{cases}
|4B_2 + 2B_3| - 2B_1, & \text{when } |2B_2 + B_3| \ge |B_3| + B_1, \\[6pt]
2|B_3|, & \text{otherwise},
\end{cases}
$$
and
$$
\Psi_{-}(p_1,p_2) \le
\begin{cases}
2B_1 - B_4, & \text{when } B_1 \ge B_4 + 2|B_3|, \\[8pt]
\dfrac{2B_1 \sqrt{2|B_3|}}{\sqrt{B_4 + 2|B_3|}}, 
& \text{when } B_1^2 \le 2|B_3|(B_4 + 2|B_3|), \\[12pt]
2|B_3| + \dfrac{B_1^2}{B_4 + 2|B_3|}, 
& \text{otherwise},
\end{cases}
$$
where
$$
B_4 = |4B_2 + 2B_3|.
$$
All inequalities are sharp.
\end{lemma}

\section{Inverse logarithmic coefficients}\label{section-3}

We now state our first main result, which provides sharp bounds of $\Gamma_1$, $\Gamma_2$ and $\Gamma_3$ for functions in $\mathcal{C}(A, B)$.

\begin{theorem} \label{P2-thm-3.1}
Let $f\in \mathcal{C}(A,B)$ be of the form \eqref{eq:1.1}. 
Then the first three inverse logarithmic coefficients satisfy the following sharp estimates
\begin{align*}
&|\Gamma_1|\le \frac{A-B}{4},\\
&|\Gamma_2|\le 
\begin{cases}
\frac{A-B}{12},& ~~~~~ |5A-B| \le 4, \\[3mm]
  \frac{(A-B)|5A-B|}{48},&~~~~~ |5A-B| > 4 ,
\end{cases}\\[2mm]
&|\Gamma_3|
\le
\begin{cases}
    \dfrac{A-B}{24}, & (\mu,\nu) \in D_1 \cup D_2, \\[3mm]
   \dfrac{(A-B)|A(3A-B)|}{48}, & (\mu,\nu) \in D_6, \\[3mm]
  \dfrac{(A- B)(|B - 5A| + 2)^{3/2}}{72 \sqrt{3(|B- 5A| + 2 + 3A^2 - AB)}}, & (\mu,\nu) \in D_8 \cup D_9.
\end{cases}
\end{align*}
\end{theorem}

\begin{proof}

Let $f \in \mathcal{C}(A, B)$ be defined by  \eqref{eq:1.1}. Then, in view of subordination, we have 
$$
1 + \frac{z f''(z)}{f'(z)} \prec \frac{1 + Az}{1 + Bz}.
$$
Thus, there exists a function $\omega \in \mathcal{B}_0$ such that 
$$
1 + \frac{z f''(z)}{f'(z)} = \frac{1 + A\omega(z)}{1 + B\omega(z)}.
$$
After  equating the coefficients of $z^2, z^3$ and $z^4$, and simplifying, we get
\begin{align}\label{eq:3.2}
a_2 &= \frac{c_1(A - B)}{2}, \\ \nonumber
a_3 &= \frac{A - B}{6} \left( (A - 2B)c_1^2 + c_2 \right),\\ \nonumber 
a_4 &= \frac{A - B}{24} \left( 2c_3 + (3A - 7B)c_1 c_2 + (A^2 - 5AB + 6B^2)c_1^3 \right). 
\end{align}
Substituting values of \eqref{eq:3.2} into  \eqref{eq:Gamma-coefficients}, we obtain 
\begin{align}
\Gamma_1 &= -\frac{c_1(A - B)}{4},  \label{eq:3.3A}\\
\Gamma_2 &= -\frac{A - B}{12} \left(c_2-\frac{(5A - B)}{4}c_1^2\right),\label{eq:3.3B}\\
\Gamma_3 &= -\frac{A - B}{24} \left( c_3 + \frac{1}{2}(B - 5A)c_1 c_2 + \frac{1}{2}(3A^2 - AB)c_1^3 \right). \label{eq:3.3}
\end{align}
Form \eqref{eq:3.3A}, we obtain
$$
|\Gamma_1| \le \frac{A - B}{4},
$$
and the inequality is sharp for inverse of the function $g_1 \in \mathcal{C}(A, B)$ defined by
\begin{equation}\label{3.4}
1 + \frac{z g_1''(z)}{g_1'(z)} = \frac{1 + Az}{1 + Bz}. 
\end{equation}
In view of  Lemma~\ref{p2-lemma-001}, form \eqref{eq:3.3B} we get
\begin{align*}
|\Gamma_2| = \frac{A - B}{12} \left| c_2 - \frac{5A - B}{4} c_1^2 \right|
& \le \frac{A - B}{12} \max \left\{ 1, \frac{|5A - B|}{4} \right\}  \\
& =
\begin{cases}
\frac{A - B}{12}, & \text{for } |5A - B| \le 4, \\[2mm]
\frac{(A - B)|5A - B|}{48}, & \text{for } |5A - B| > 4.
\end{cases}
\end{align*}
The first inequality is sharp for $g_2 \in \mathcal{C}(A, B)$ defined by
\begin{equation}\label{equ 3.5}
1 + \frac{z g_2''(z)}{g_2'(z)} = \frac{1 + Az^2}{1 + Bz^2}. 
\end{equation}
and  second inequality is sharp for inverse of the function $g_1 \in \mathcal{C}(A, B)$ defined by \eqref{3.4}.

Further form \eqref{eq:3.3} we get 
\begin{align}
    |\Gamma_3| & = \frac{A - B}{24} \left| c_3 + \frac{1}{2}(B - 5A)c_1 c_2 + \frac{1}{2}(3A^2 - AB)c_1^3 \right| \notag\\  
    & = \frac{A - B}{24} \left| c_3 + \mu c_1 c_2 + \nu c_1^3 \right| \label{eq:3.4}.
\end{align}
where 
\begin{equation}\label{eq:3.41}
   \mu = \frac{(B - 5A)}{2}, \quad \nu = \frac{(3A^2 - AB)}{2}. 
\end{equation}
We now apply Lemma \ref{p2-lemma-002} to estimate $\Gamma_3$. For $-1 \le B < A \le 1$, one can check that
\begin{equation}\label{eq:3.4A}
    \mu \in [-3, 2), \quad \nu \in \left[-\frac{1}{24}, 2\right].
\end{equation}

The sets $D_k  ~~ \mbox{for}~~~k=1,\dots,12$ are defined as in Lemma~\ref{p2-lemma-002} above.

\begin{itemize}
\item $D_3$ is empty as $\nu \in \left[-\frac{1}{24},2\right]$.
\item If $|\mu| \ge \frac{1}{2}$, then $-\frac{2}{3}(|\mu| + 1) \le -1,$ and Thus, the condition $\nu \le -\frac{2}{3}(|\mu| + 1) \leq -1$ contradicts the admissible range of $\nu$ given in \eqref{eq:3.4A}. Therefore, the set $D_4$ is empty.

\item The conditions for $D_5$ ($|\mu| \le 2$ and $\nu \ge 1$) are equivalent to
$$
\frac{B-4}{5} \le A \le \frac{B - \sqrt{B^2 + 24}}{6}, \quad |B| < 1.
$$
A detail calculation shows that  for any $|B| < 1$, we have
$\frac{B-4}{5} > \frac{B - \sqrt{B^2 + 24}}{6}.$
Hence, the set $D_5$ is also empty.
\item Since $|\mu| \le 4$, it follows that the sets $D_7, D_{11}$, and $D_{12}$ are empty.  

\item When $2 \le |\mu| \le 4$, we must have $|\mu| \in (2, 3]$. In this case,
$$
\frac{2|\mu|(|\mu| + 1)}{\mu^2 + 2|\mu| + 4}>\frac{\mu^2 +2}{12}.
$$
This shows that $D_{10}$ is also empty.
\end{itemize}
Also, we can check that,
\begin{itemize}
\item when $(A, B) = (0.2, 0)$, then the related $(\mu, \nu) \in D_1$,
\item when $(A, B) = (-0.5, -0.6)$,                                                                                                                                                          the associated $(\mu, \nu) \in D_2$,
\item when $( A , B) = (1 , -1)$ , then the associated $(\mu, \nu) \in D_6$.
\item when $(A, B) = ( -0.795 , -1 )$, then the associated $(\mu, \nu) \in D_8$,
\item when $(A, B) = ( 0.6 , -1 )$, then the associated $(\mu, \nu) \in D_9$.
\end{itemize}
This shows that the sets $D_1, D_2, D_6, D_8$, and $D_9$ are nonempty and that the entire range of possible values of $(A, B) $ is covered.  
Now we consider the following cases:\\[3mm]

\noindent \textbf{Case-I}: Let $(\mu, \nu) \in D_1 \cup D_2$. Then by  Lemma~\ref{p2-lemma-002}, form \eqref{eq:3.4} we get 
$$
|\Gamma_3| \le \frac{A - B}{24}.
$$
The inequality is sharp for inverse of the function $g_3 \in \mathcal{C}(A, B)$ defined by
$$
1 + \frac{z g_3''(z)}{g_3'(z)} = \frac{1 + Az^3}{1 + Bz^3}.
$$
For the function $g_3 \in \mathcal{C}(A, B)$, it is easy to check that
$$
a_2 = 0, \quad a_3 = 0, \quad a_4 = \frac{A - B}{12}. 
$$
Therefore, for the function $g_3$, we get 
$$
|\Gamma_3 | = \left|\frac12\Big(-a_4+4a_2a_3-\frac{10}{3}a_2^3\Big)\right|
= \frac{A - B}{24}.
$$

\noindent \textbf{Case-II}: Let $(\mu, \nu) \in D_6$. Then, using Lemma~\ref{p2-lemma-002}, form \eqref{eq:3.4} we get  we obtain 
$$
|\Gamma_3| \le \frac{(A-B)|A(3A - B)|}{48} .
$$
A detailed calculation shows that the above inequality is sharp for the inverse of the function $g_1 \in \mathcal{C}(A, B)$ defined by \eqref{3.4}. For the function $g_1 \in \mathcal{C}(A, B)$ we have 
$$
a_2 = \frac{A - B}{2}, \quad
a_3 = \frac{A^2 - 3AB + 2B^2}{6}, \quad
a_4 = \frac{A^3 - 5A^2B + 11AB^2 - 6B^3}{24},
$$
and therefore 
$$
|\Gamma_3| = \left|\frac12\Big(-a_4+4a_2a_3-\frac{10}{3}a_2^3\Big) \right| = \frac{(A-B)|A(3A - B)|}{48}.
$$

\noindent \textbf{Case-III}: Let $(\mu, \nu) \in D_8 \cup D_9$. Then in view of Lemma~\ref{p2-lemma-002}, form \eqref{eq:3.4}
\begin{align*}
  |\Gamma_3| \le   \frac{2(A - B)(|\mu| + 1)^{3/2}}{72 \sqrt{3(|\mu| + 1 + \nu)}},
\end{align*}
where $\mu$ and $\nu$ are defined by \eqref{eq:3.41}. 
By putting the value of $\mu$ and $\nu$, we get 
$$
|\Gamma_3| \leq  \frac{(A- B)(|B - 5A| + 2)^{3/2}}{72 \sqrt{3(|B- 5A| + 2 + 3A^2 - AB)}}.
$$
The above inequality is sharp for the function $g_4 \in \mathcal{C}(A, B)$ defined by 
$$
1 + \frac{z g_4''(z)}{g_4'(z)} = \frac{1 + A\omega(z)}{1 + B\omega(z)},
$$
where $\omega \in \mathcal{B}_0$  is defined as 
$$
\omega(z) = \frac{z(c - \operatorname{sgn}(\mu) z)}{(1 - \operatorname{sgn}(\mu) c z)}, \qquad c = \left( \frac{|\mu| + 1}{3(|\mu| + \nu + 1)} \right)^{1/2}.
$$
Here it is noted that for 
$(\mu, \nu) \in D_8 \cup D_9$, we have  $\nu \ge -2 (|\mu|+1) /3$. This implies that,
$$ 3(|\mu| + \nu + 1) \ge 3(|\mu| + 1) - 2(|\mu| + 1) = (|\mu| + 1) $$ and hence $|c|^2 \leq 1.$  If  $|c|^2 = 1 $, we get 
$$
\nu= \frac{-2 (|\mu| + 1)}{3}.
$$
As $(\mu, \nu) \in (D_8 \cup D_9)$ and $ \mu \in  [-3, 2)$, a detail calculation shows that   $\nu$ must lies in $(-2, -1]$, which contradicts the range of $\nu$ defined by \eqref{eq:3.4A}. Therefore for $(\mu, \nu) \in D_8 \cup D_9$,  we get $|c| < 1$ and hence 
$$
\omega(z) = \frac{z(c - \operatorname{sgn}(\mu) z)}{(1 - \operatorname{sgn}(\mu) c z)} \in \mathcal{B}_0.
$$
For inverse of the function $g_4 \in \mathcal{C}(A,B)$, from \eqref{eq:3.4}, we have
\begin{align*}
\left|\Gamma_3\!\left(g_4^{-1}\right)\right|
&= \frac{A-B}{24}
\left|c_3+\mu c_1c_2+\nu c_1^3\right| \\
&= \frac{A-B}{24}
\left|c(c^2-1)+\mu c(c^2-1)\operatorname{sgn}(\mu)+\nu c^3\right| \\
&= \frac{c(A-B)}{36}(|\mu|+1).
\end{align*}
By putting the value of $\mu$ and $\nu$ defined by \eqref{eq:3.41}, we get 
$$
\left|\Gamma_3\!\left(g_4^{-1}\right)\right| =  \frac{(A- B)(|B - 5A| + 2)^{3/2}}{72 \sqrt{3(|B- 5A| + 2 + 3A^2 - AB)}}.
$$

\end{proof}


\section{Moduli Difference of Successive Logarithmic and Logarithmic Inverse Coefficients}\label{section-4}

In this section, we obtain the sharp lower and upper bounds of  $\bigl|\,\gamma_2 \,\bigr|-\bigl|\,\gamma_1\,\bigr|$
and $\bigl|\,\Gamma_2 \,\bigr|-\bigl|\,\Gamma_1\,\bigr|$ for functions in $\mathcal{C}(A, B)$.

\begin{theorem}\label{P2-thm-4.1}
 Let $f\in \mathcal{C}(A,B)$ be defined by \eqref{eq:1.1}.  Then the following sharp inequalities hold.
 \begin{align*}
   &\bigl|\,\gamma_2 \,\bigr|-\bigl|\,\gamma_1\,\bigr|\le\dfrac{(A-B)}{12}, \\[3mm]
  &\bigl|\,\gamma_2 \,\bigr|-\bigl|\,\gamma_1\,\bigr|\ge
 \begin{cases}
 -\dfrac{(A-B)}{4}\bigl(1-\dfrac{|A-5B|}{12} \bigr), &  |A-5B|\leq 2,\\[2mm]
 -\dfrac{(A-B)}{2}\dfrac{1}{\sqrt{|A-5B|+4}}, &|A-5B|\ge5,\\[2mm]
-\dfrac{(A-B)}{12}\bigl(1+\dfrac{9}{|A-5B|+4}\bigr), &\textit{otherwise}.
\end{cases}  
 \end{align*}

\end{theorem}
\begin{proof}
 Let $f(z)\in \mathcal{C}(A,B)$ be of the form \eqref{eq:1.1}. So in view of subordination, there exists a function $p \in \mathcal{P}$ defined by \eqref{P Class-1} such that, 
\begin{equation}\label{P2-eq-40}
   1 + \frac{z f''(z)}{f'(z)} = \frac{(1-A) + (1+A)p(z)}{(1-B) + (1+B)p(z)}. 
\end{equation}
 By equating the coefficients of $z^2$ and $z^3$ in \eqref{P2-eq-40}, we get 
\begin{align}\label{P2-eq-45}
  & a_2 = \frac{(A - B)}{4}p_1  \\ \nonumber 
  & a_3 = \frac{(A - B)}{6} \left(\frac{p_2}{2}  + \frac{(A - 2B - 1)}{4} p_1^2\right).
\end{align}
In view of \eqref{P2-eq-45}, form (\ref{eq:gamma-coefficients}) we get,
\begin{align}\label{P2-eq-50}
  \bigl|\,\gamma_2 \,\bigr|-\bigl|\,\gamma_1\,\bigr| &= \left|\frac12\Big(a_3-\frac12 a_2^2\Big) \right| - \left| \frac{a_2}{2}\right|\\ \nonumber
  &= \left|\frac{(A-B)(A-5B-4)}{192}\,p_1^2+\frac{(A-B)}{24}\,p_2\,\right|-\left|\,\frac{(A-B)}{8}\,p_1\right|\\ \nonumber 
&= |B_2 p_1^2 + B_3 p_2| - |B_1 p_1|  = \Psi_{+}(p_1,p_2).
\end{align}
Where 
\begin{equation}\label{P2-eq-55}
 B_1=\frac{(A-B)}{8},~~B_2=\frac{(A-B)(A-5B-4)}{192}~~ \mbox{and}~~ B_3=\frac{(A-B)}{24}.   
\end{equation}
After doing a detail calculation, for $-1 \leq B < A \leq 1$, we get
$$|2B_2+B_3|\ngeq|B_3|+B_1.$$
Thus by Lemma~\ref{p1-lemma-003}, we obtain 
$$\Psi_{+}(p_1,p_2) \le2|B_3|.$$
Hence from \eqref{P2-eq-50}, we get 
$$\bigl|\,\gamma_2 \,\bigr|-\bigl|\,\gamma_1\,\bigr|\le \frac{(A-B)}{12}.$$
The inequality is sharp for the function $g_2$ defined in (\ref{equ 3.5}).

Now, to get the lower estimation  of $\bigl|\,\gamma_2 \,\bigr|-\bigl|\,\gamma_1\,\bigr|$, we consider 
\begin{equation}\label{P2-eq-60}
  B_4=|4B_2+2B_3|
=\frac{(A-B)|A-5B|}{48}.  
\end{equation}
A detailed calculation shows that  for $-1 \leq B < A \leq 1$ 
\begin{equation*}
    B_4 + 2 |B_3| = \frac{(A- B)(|A - 5 B| + 4 )}{8}. 
\end{equation*}
and  the inequality  $B_1 \geq B_4+2|B_3|$ holds for $|A-5B|\leq 2.$\\ Also for  $-1 \leq B < A \leq 1$,  the inequality $ B_1^2 \leq 2|B_3|(B_4+2|B_3|)$ is equivalent to $|A-5B|\ge 5.$\\ So by Lemma \ref{p1-lemma-003}, we obtain 
$$
\Psi_{-}(p_1,p_2) \le
\begin{cases}
2B_1 - B_4, & \text{when} |A-5B|\leq 2, \\[8pt]
\dfrac{2B_1 \sqrt{2|B_3|}}{\sqrt{B_4 + 2|B_3|}}, 
& \text{when } |A-5B|\ge 5,\\[3mm]
2|B_3| + \dfrac{B_1^2}{B_4 + 2|B_3|}, 
& \text{otherwise},
\end{cases}
$$
where $B_1, B_2, B_3$ and $B_4$ are defined by \eqref{P2-eq-55} and \eqref{P2-eq-60}.
As form Lemma \ref{p1-lemma-003}, we know $\Psi_{-}(p_1,p_2) = -\Psi_{+}(p_1,p_2)$, thus by putting the values of $B_1, B_2, B_3$ and $B_4$, we get  
\begin{align*}
  &\bigl|\,\gamma_2 \,\bigr|-\bigl|\,\gamma_1\,\bigr|\ge
 \begin{cases}
 -\dfrac{(A-B)}{4}\bigl(1-\dfrac{|A-5B|}{12} \bigr), &  |A-5B|\leq 2,\\[2mm]
 -\dfrac{(A-B)}{2}\dfrac{1}{\sqrt{|A-5B|+4}}, &|A-5B|\ge5,\\[2mm]
-\dfrac{(A-B)}{12}\bigl(1+\dfrac{9}{|A-5B|+4}\bigr), &\textit{otherwise}.
\end{cases}  
 \end{align*}
The inequalities are sharp for the functions $g_5\in \mathcal{C}(A,B)$ defined by 
$$
1 + \frac{z g_5''(z)}{g_5'(z)} = \frac{(1-A) + (1+A)p_1(z)}{(1-B) + (1+B)p_1(z)},
$$
where $p_1(z) \in \mathcal{P}$ with the following form 
 $$p_1(z)=\begin{cases}
\dfrac{1+z}{1-z} & \mbox{for}~~ |A-5B|\le2,\\[2mm]
\dfrac{1-z^2}{1-{\dfrac{4}{\sqrt{|A-5B|+4}}}\,z+z^2} & \mbox{for}~~ |A-5B|\ge 5.\\[2mm]
\dfrac{1-z^2}{1-{\dfrac{12}{{|A-5B|+4}}}\,z+z^2} & \mbox{otherwise.}
 \end{cases}$$

\end{proof}

\begin{remark}
   By taking $A=1$ and $B=-1$ in Theorem~\ref{P2-thm-4.1}, we obtain the result by Kumar and Cho \cite{KumarCho2023} for function $f\in \mathcal C$
   $$
-\frac{1}{\sqrt{10}}
\le
|\gamma_2|-|\gamma_1|
\le
\frac1{6}.
$$
as a particular case. 
\end{remark}

Next, we obtain the sharp upper and lower bounds of $|\Gamma_2| - |\Gamma_1|$ for functions in $\mathcal{C}(A, B)$. 

\begin{theorem}\label{P2-thm-4.3}
 Let $f\in \mathcal{C}(A,B)$ be of the form \eqref{eq:1.1} and $\Gamma_1$ and $\Gamma_2$ are logarithmic inverse coefficients of $f$. Then 
 \begin{align*}
&\bigl|\,\Gamma_2 \,\bigr|-\bigl|\,\Gamma_1\,\bigr|\le\frac{(A-B)}{12},\\[3mm]
&\bigl|\,\Gamma_1 \,\bigr|-\bigl|\,\Gamma_2\,\bigr|\ge 
 \begin{cases}
  \dfrac{(A-B)}{4}\bigl(1-\frac{|5A-B|}{12} \bigr) &  |5A-B| \leq 2\\[2mm]
 \dfrac{(A-B)}{2 \sqrt{|5A-B|+4}}  & |5A-B|\ge5\\[2mm]
 \dfrac{(A-B)}{12}\left(1+\dfrac{9}{|5A-B|+4}\right) & ~~~~~\mbox{otherwise.}
\end{cases} 
\end{align*}
All these inequalities are sharp.
\end{theorem}

\begin{proof}
As $f\in \mathcal{C}(A,B)$, in view of \eqref{P2-eq-40} and \eqref{P2-eq-45}, we get 
\begin{align}\label{P2-eq-65}
    \bigl|\,\Gamma_2 \,\bigr|-\bigl|\,\Gamma_1\,\bigr|& =\bigl|\,\frac{(A-B)(5A-B+4)}{192}\,p_1^2-\frac{(A-B)}{24}\,p_2\,\bigr|-\bigl|\,\frac{(A-B)}{8}\,p_1\,\bigr|\\ \nonumber 
    & = |B_2 p_1^2 + B_3 p_2| - |B_1 p_1| = \Psi_{+}(p_1,p_2)
\end{align}
Where $$B_1=\frac{(A-B)}{8},~~B_2=\frac{(A-B)(5A-B+4)}{192} ~~ \mbox{and}~~ B_3=-\frac{(A-B)}{24}.$$
For $-1 \leq B < A \leq 1$, it is straightforward to check that 
$$|2B_2+B_3|\ngeq|B_3|+B_1. $$
So using Lemma~\ref{p1-lemma-003}, we obtain $$\Psi_{+}(p_1,p_2) \le2|B_3|. $$
Therefore form \eqref{P2-eq-65}, we get 
$$\bigl|\,\Gamma_2 \,\bigr|-\bigl|\,\Gamma_1\,\bigr|\le\frac{(A-B)}{12}.$$
The inequality is sharp for the inverse of the function $g_1 \in \mathcal{C}(A, B)$ defined in (\ref{3.4}).

Now, to derive a lower bound for $ \bigl|\,\Gamma_2 \,\bigr|-\bigl|\,\Gamma_1\,\bigr|,$
We consider,
$$B_4=|4B_2+2B_3|=\frac{(A-B)|5A-B|}{48},$$
A detailed calculation shows that  for $-1 \leq B < A \leq 1$ 
  the inequalities   $B_1 \geq B_4+2|B_3|$ holds for $ |5A - B|\leq 2 $ and   $ B_1^2 \leq 2|B_3|(B_4+2|B_3|)$ is satisfied for $|5A - B|\ge 5.$ So by Lemma \ref{p1-lemma-003}, we obtain
$$\Psi_{-}(p_1,p_2) \le 
\begin{cases}
 2B_1 - B_4 & |5A-B|\le2 \\ 
\dfrac{2B_1 \sqrt{2|B_3|}}{\sqrt{B_4 + 2|B_3|}} 
&  |5A-B|\ge5\\ 
2|B_3| + \dfrac{B_1^2}{B_4 + 2|B_3|}
& \text{otherwise}. 
\end{cases},$$
From Lemma \ref{p1-lemma-003}, we know  $\Psi_{-}(p_1,p_2) = -\Psi_{+}(p_1,p_2).$ Therefore by putting the values of $B_1, B_2, B_3$ and $B_4$ we get  

\begin{align*}
  &\bigl|\,\Gamma_1 \,\bigr|-\bigl|\,\Gamma_2\,\bigr| \le
 \begin{cases}
 \dfrac{(A-B)}{4}\left(1-\dfrac{|5A-B|}{12} \right)  &  |5A-B| \leq 2\\[2mm]
 \dfrac{(A-B)}{2}\dfrac{1}{\sqrt{|5A-B|+4}}  &|5A-B|\ge5\\[2mm]
 \dfrac{(A-B)}{12}\left(1+\dfrac{9}{|5A-B|+4}\right) & ~~~~~\mbox{otherwise}.
\end{cases}  
 \end{align*}
The inequalities are sharp for the functions $g_6\in \mathcal{C}(A,B)$ defined by 
$$
1 + \frac{z g_6''(z)}{g_6'(z)} = \frac{(1-A) + (1+A)p_2(z)}{(1-B) + (1+B)p_2(z)}
.$$
where $p_2(z) \in \mathcal{P}$ with the following form 
 $$p_2(z)=\begin{cases}
\dfrac{1+z}{1-z} & |5A-B|\le 2,\\[2mm]
\dfrac{1-z^2}{1-{\frac{4}{\sqrt{|5A-B|+4}}}\,z+z^2} & {|5A-B|\ge5},\\[2mm]
\dfrac{1-z^2}{1-{\frac{12}{{|5A-B|+4}}}\,z+z^2} &\mbox{otherwise.}
 \end{cases}$$
\end{proof}

\begin{remark}
    Here we want to mention that in 2025, Allu and Shaji \cite{Allu2025} obtained the sharp upper and lower bounds of $|\Gamma_2|-|\Gamma_1|$ for functions $f \in \mathcal{C}(\alpha)$ where  $0 \le \alpha < 1$. By taking $A=1-2\alpha$ and $B=-1$ in Theorem~\ref{P2-thm-4.3}, we obtain  the results in  Allu and Shaji \cite{Allu2025} (Corollary 3.5 and Corollary 3.11) as a particular case.
\end{remark}


\section{ Hankel Determinant Having Elements of Logarithmic
Inverse Coefficient}\label{section-5}
In this section, we discuss the second Hankel determinant of $F_{f^{-1}}/2$ for functions in the class $\mathcal{C}(A, B)$ where $ -1 \leq B < A \leq 1$. The second Hankel determinant of $F_{f^{-1}}/2$ using \eqref{eq:Gamma-coefficients}, is given by 
\begin{equation}\label{P1-eq-mr-136}
    H_{2,1}\!\left(\frac{F_{{f}^{-1}}}{2}\right) :=  \Gamma_3 \Gamma_1 - \Gamma_2^2 =\dfrac{1}{4}\left(\dfrac{13}{12}a_2^{4}+ a_2a_4-  a_2^2a_3- a_3^2\right).
\end{equation}
In the next result, we get the bound of $H_{2,1}(F_{{f}^{-1}}/2)$ for functions in $\mathcal{C}(A, B)$.

\begin{theorem}
Let $f\in\mathcal{C}(A,B)$  be of the form \eqref{eq:1.1}. Then
$$
\left| H_{2,1}\!\left(\frac{F_{f^{-1}}}{2}\right) \right|\le \begin{cases} 
\frac{(A - B)^2(11A^2-2AB -B^2)- 48|5A-B|-4(5A-B)^2-112}{144( (11A^2-2AB-B^2)-4|5A-B|-8)}\quad \mbox{for}~~~ |5A-B|\ge2\\[3mm]
  \frac{(A-B)^2}{144} \qquad  \mbox{for}~~~ |5A-B|<2. 
\end{cases}
$$
The last inequality is sharp.
\end{theorem}

\begin{proof}
Let $f \in \mathcal{C}(A,B) $. 
Then the second Hankel determinant of $f$ is given by  
 \begin{equation}\label{eq:5.1}
 H_{2,1}(F_{f^{-1}}/2)
=\Gamma_1\Gamma_3-\Gamma_2^2.
  \end{equation}
Using the expressions for $\Gamma_1, \Gamma_2, \Gamma_3$
given in \eqref{eq:Gamma-coefficients}, equation \eqref{eq:5.1} reduces to
\begin{equation}\label{eq:5.01}
H_{2,1}(F_{f^{-1}}/2)=\dfrac{13}{48}a_2^{4}+\dfrac{1}{4}a_2a_4-\dfrac{1}{4}a_2^2a_3-\dfrac{1}{4}a_3^2.    
\end{equation}
Further using \eqref{eq:3.3A}, \eqref{eq:3.3B} and (\ref{eq:3.3}) in \eqref{eq:5.01}, we get    
\begin{equation}
\begin{aligned}\label{eq:5.2}
H_{2,1}(F_{f^{-1}}/2)
&=\frac{1}{2304}\Big((11A^2-2AB-B^2)(A-B)^2c_1^4
-4(A-B)^2(5A-B)c_1^2c_2 \\
&\qquad\qquad~~~~~~~~~~~~~~~~~~~~~~~~~~~~~~~~~~~~~~~
+24(A-B)^2c_1c_3
-16(A-B)^2c_2^2\Big) \\
&=\frac{1}{2304}(A-B)^2\Big((11A^2-2AB-B^2)c_1^4
-4(5A-B)c_1^2c_2 \\
&\qquad\qquad
+24c_1c_3
-16c_2^2\Big).
\end{aligned}
\end{equation}
Using  Lemma~\ref{lemma-4},  (\ref{eq:5.2}) can be rewritten as  
\begin{align}\label{eq:5.3}
H_{2,1}(F_{f^{-1}}/2)=
&\frac{(A-B)^2}{2304}\Big((11A^2-2AB-B^2)c_1^4-4(5A-B)c_1^2(1-c_1^2)x\\\nonumber 
&-8(1-c_1^2)x^2(2+c_1^2) + 24c_1(1-c_1^2)(1-|x|^2)s\Big).
\end{align}
Since both $H_{2,1}(F_{f^{-1}}/2)$ and the class $\mathcal{B}_{0}$ are invariant under rotations, we may assume, without loss of generality, that $c_1$ is real. \\By taking $u:=c_1\in [0 , 1 ]$ and $v:=|x|\in [0,1],$ and using  triangle inequality, the  expression (\ref{eq:5.3}) yields,
\begin{equation}\label{eq:5.4}
\begin{aligned}
\left| H_{2,1}\!\left(\frac{F_{f^{-1}}}{2}\right) \right|
\le  & \frac{(A-B)^2}{2304}\Bigl(
\left(11A^2-2AB-B^2 \right)u^4
+4|5A-B|u^2(1-u^2)v \\
& \qquad
+8(2+u^2)(1-u^2)v^2
+24u(1-u^2)(1-v^2)
\Bigr) \\
=  & \frac{(A-B)^2}{2304}\Bigl(
8(1-u^2)(1-u)(2-u)v^2
+4|5A-B|u^2(1-u^2)v \\
& \qquad
+\left(11A^2-2AB-B^2 \right)u^4
+24u(1-u^2)
\Bigr) =  \mathcal{J}_u(v).
\end{aligned}
\end{equation}
Since $\mathcal{J}_u$ defined by  (\ref{eq:5.4}) is increasing with respect to $v$, the maximum value of $\mathcal{J}_u$ occurs at the point $v=1$.\\
Therefore we get 
$$
\mathcal{J}_u(1)=\max\{J_u(v): 0\le v\le1\}
=\mathcal{L}(u) \textit{(say).}
$$
Where 
\begin{align}
\mathcal{L}(u) 
=  & \frac{(A-B)^2}{2304}\Bigl( 8(1-u^2)(u^2+2) \nonumber\\
& +4|5A-B|u^2(1-u^2)  +\left(11A^2-2AB-B^2 \right)u^4 \Bigr). \label{eq:5.5}
\end{align}
By putting  $u^2=t$, we rewrite the equation (\ref{eq:5.5}) as 
\begin{equation}\label{eq:5.6}
\mathcal{L}(t) = \frac{(A-B)^2}{2304}\left(Wt^2+Xt+ 16 \right) ~~\mbox{for}~~ 0 \le t\le1.
\end{equation}
Where $W = 11A^2-2AB-B^2-4|5A-B|-8$ and $ X = 4|5A-B|- 8 $. 

We need to find the maximum value of $\mathcal{L}$ defined by \eqref{eq:5.6} for $0\le t\le1$.  Let  consider 
\begin{equation}\label{eq:5.6A}
 Q(t)=Wt^2+Xt+ 16, ~~~~~ 0 \le t\le1.   
\end{equation}
Now we consider the following cases: 

\noindent \textbf{Case-I}:  If $W>0$, for $0 \leq t \leq 1$, from \eqref{eq:5.6A} we get 
$$
\max_{0\le t\le1} Q(t) = \max\{16 ,\;W+X+16\}.
$$
\noindent \textbf{Case-II}: If $W<0$, the parabola defined by \eqref{eq:5.6A} opens downward. Thus, the critical point of 
$Q$ lies in $( 0, 1 )$. A straightforward calculation shows that the critical point is at 
$$
t_0=-\frac{X}{2W}.
$$
So we obtain 
$$
\max_{0\le t\le1} Q(t) = \begin{cases}
16 -\dfrac{X^2}{4W},
& 0< t_0 <1, \\[2ex]
\max\{16,\;W+X+16\},
& \text{otherwise}.
\end{cases}
$$
\noindent \textbf{Case-III}: If $W = 0$, for $0 \leq t \leq 1$, from \eqref{eq:5.6A} we get
$$
\max_{0\le t\le1} Q(t) = \max\{16 ,\; X+ 16\}.
$$

Combining all the above cases, we have 
\begin{equation}\label{eq:5.7A}
 \max_{0 \le t \le 1} Q(t)  =   \begin{cases}
16-\dfrac{X^2}{4 W}, 
\quad W<0 \ \text{and} \ 0 < -\dfrac{X}{2W} < 1, \\ 
\max\{16,\;W+X+16\},
\quad  \ \text {otherwise} .
\end{cases}
\end{equation}
By substituting the values of $W$, $X$ in \eqref{eq:5.7A} and  in view of \eqref{eq:5.5}, from \eqref{eq:5.4}, we get 
\begin{equation}\label{eq:5.7}
\left| H_{2,1}\!\left(\frac{F_{f^{-1}}}{2}\right) \right|\le \begin{cases}\,\frac{(A-B)^2}{2304}\frac{16(11A^2-2AB -B^2)- 48|5A-B|-4(5A-B)^2-112}{[(11A^2-2AB-B^2)-4|5A-B|-8]},\qquad~~~~~~~ |5A-B| > 2,\\[3mm]
  \frac{(A-B)^2}{144}, \quad |5A-B|\le2.
\end{cases}
\end{equation}
The second inequality in (\ref{eq:5.7}) is sharp for the inverse of the function $g_2 \in \mathcal{C}(A, B)$ defined in (\ref{equ 3.5}).

\end{proof}

\section{Declarations}
\textit{Acknowledgements :} The first author would like to thank the UGC, Govt. of India, for the financial support (NTA Ref. No. 231620097064) in the form of a fellowship.\\
\textit{Author Contributions:} All the authors contributed equally to this manuscript and reviewed it. \\[2mm]
 \textit{Data Availability :} No datasets were generated or analysed during the current study. \\
\textit{Conflict of interest:} There is no competing interest.\\

\bibliographystyle{amsplain}
\bibliography{references}

\end{document}